\documentclass[10 pt]{amsart}
\usepackage{amssymb}
\usepackage{amsmath}
\usepackage{amscd}
\usepackage{graphicx}
\usepackage{latexsym}

\calclayout\evensidemargin .6in

\newcommand{\au}{$Aut(\Sigma,\partial \Sigma) \,\,$}
\newcommand{\aut}{$Aut(\Sigma,\partial \Sigma)$}   

\newcommand{\rv}{right-veering }

\def\R{\mathbb{R}}
\def\Z{\mathbb{Z}}

\newtheorem{theorem}{Theorem}[section]
\newtheorem{lemma}[theorem]{Lemma}

\newtheorem{corollary}[theorem]{Corollary}
\newtheorem{definition}[theorem]{Definition}

\theoremstyle{remark}
\newtheorem{remark}[theorem]{\rm\bfseries{Remark}}
\def\dfn#1{{\em #1}}

\title[On the Classification of Planar Contact Structures]
{On the Classification of Planar Contact Structures}

\author{M. Firat Arikan}
\address{Department of Mathematics, University of Rochester, Rochester NY 14627, USA}
\email{arikan@math.rochester.edu}

\author{Selahi Durusoy}
\address{Department of Mathematics, Grand Valley State University, Grand Rapids MI 49401, USA}
\email{durusoyd@gvsu.edu}

\begin{document}
\begin{abstract}

In this paper, we focus on contact structures supported by planar
open book decompositions. We study right-veering diffeomorphisms
to keep track of overtwistedness property of contact  structures under
some monodromy changes. As an application we give infinitely many examples of overtwisted
contact structures supported by open books whose pages
are the four-punctured sphere, and also we prove that a certain
family is holomorphically fillable using lantern relation.
\end{abstract}

\maketitle

%
%
\section{Introduction} \label{sec:intro}

Let $(M,\xi)$ be a closed oriented 3-manifold with the contact
structure $\xi$, and let $(S,h)$ be an open book (decomposition)
of $M$ which is compatible with $\xi$. (We refer the
reader to \cite{Et3, Ge} for contact geometry, and to \cite{Et2, Gd} for open books and compatibility). Based on the correspondence given in \cite{Gi}, two topological invariants were defined in \cite{EO}:
\[
sg(\xi)=\min \{ \, \,g(S) \, \, \vert \, \,(S,h) \text{ an open
book supporting } \xi\},
\]
called the \dfn{support genus} of $\xi$, and
\[
bn(\xi)=\min \{ \, \, |\partial S|  \, \,\vert  \, \, (S,h) \text{
an open book supporting }  \xi \text{ and } g(S)=sg(\xi)\},
\]
called the \dfn{binding number} of $\xi$. It is proved in \cite{Et1}
that if $(M,\xi)$ is overtwisted, then $sg(\xi)=0$, i.e., $\xi$
is supported by a planar open book. The algorithm
given in \cite{Ar1} finds a reasonable upper bound for $sg(\xi)$
using the given contact surgery diagram of $\xi$. On the other hand,
for $sg(\xi)=0$ and $bn(\xi)\leq2$, the complete list of all such
planar contact structures (up to isotopy) is given in \cite{EO}. The case
when $sg(\xi)=0$ and $bn(\xi)=3$ is classified in \cite{Ar2}. In the
last two mentioned papers, it was also shown which structures are
tight and which ones are overtwisted. As an application of the techniques
which we will develop, we will partially consider the case where
$sg(\xi)=0$ and $bn (\xi)\leq4$ at the end of this paper. We refer the reader \cite{L} for tight contact structures supported by four-punctured sphere.

The structure of this paper is as follows: In Section \ref{sec:intro},
we state the theorems that we will prove later in the paper.
After the preliminaries (Section \ref{sec:preliminaries}), right-veering
diffeomorphisms and overtwisted planar contact structures are studied in
Section \ref{sec:ot} where we also give alternative proofs of some results
recently proved in \cite{Y} (see Lemma \ref{lem:OT+Left_twist}, Remark
\ref{rem:Yilmaz}, Lemma \ref{lem:Conjugate_of_RV_is_RV}). In Section \ref{sec:4-Holed},
we focus on the four-punctured sphere and prove our main results.

Let $D_{\gamma}$ denote the right Dehn twist
along the simple closed curve $\gamma$. Most of the time we'll
write $\gamma$ instead $D_\gamma$ for simplicity. For any
bordered surface $S$, let $Aut(S,\partial S)$ be the group of
isotopy classes of orientation preserving diffeomorphisms of $S$
which restrict to identity on $\partial S$. In $Aut(S,\partial S)$, we
will multiply a new element from the right of the previously given
word although we compose the corresponding difeomorphisms of $S$
from left.

For a given fixed open book $(S,h)$ of a $3$-manifold $M$, by \cite{Gi}, $(S,h)$
determines a unique contact manifold $(M_{(S,h)},\xi_{(S,h)})$ up to
contactomorphism. We will shorten the notation as $(M_h,\xi_h)$ if
the surface $S$ is clear from the content.

Let $\Sigma$ be the four$-$punctured sphere obtained by deleting the
interiors of four disks from the $2-$sphere $S^2$ (see Figure
\ref{4-holed_1}). Let $C_1, C_2, C_3, C_4$ be the boundary
components of $\Sigma$, and let $a, b, c, d$ denote the simple
closed curves parallel to the boundary components $C_1, C_2, C_3,
C_4$, respectively. Also consider the simple closed curves $e, f, g,
h$ in $\Sigma$ given as in Figure \ref{4-holed_1}.

\begin{figure}[ht]
   \includegraphics{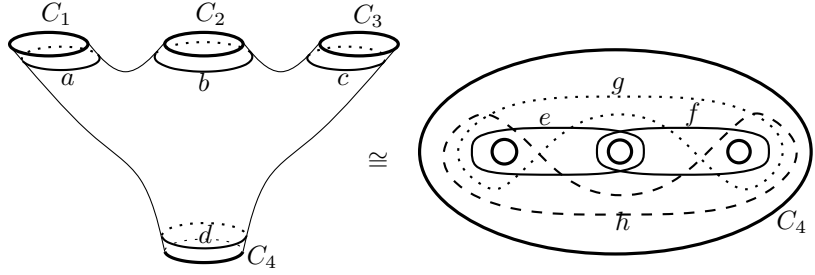}
   \caption{Four$-$punctured sphere $\Sigma$, and the simple closed curves.}
  \label{4-holed_1}
\end{figure}

Let $\phi \in$ \au be any element. In Section \ref{sec:4-Holed}, it
will be clear that we can write
$$\phi=a^{r_1}b^{r_2}c^{r_3}d^{r_4}e^{m_1}f^{n_1}\cdots
e^{m_s}f^{n_s}$$ for some integers $m_i$'s and $n_i$'s (see Lemma
\ref{lem:reducing_momodromy}). Our main results are the following:

\begin{theorem} \label{thm:holofillable}
The contact manifold $(M_\phi, \xi_{\phi})$
is holomorphically fillable in each of the following cases:
\begin{enumerate}
\item[(H1)] $s=1, max\{m_1,n_1\} \ge 0, min\{r_k\} \ge max \{ -m_1,-n_1,0\}$,
\item[(H2)] $s=1, m_1<0, n_1<0, max\{m_1,n_1\}=-1, min\{r_k\} \ge -m_1-n_1-1$,
\item[(H3)] $s=1, m_1<0, n_1<0, max\{m_1,n_1\}<-1, min\{r_k\} \ge -m_1-n_1-2$,
\item[(H4)] $s>1, min\{r_k\} \ge \sum_{i=1}^s max\{-m_i,0\} + \sum_{j=1}^s max\{-n_j,0\}$.
\end{enumerate}
\end{theorem}

For the other results, we focus only on the elements of the form $\phi=a^{r_1}b^{r_2}c^{r_3}d^{r_4}e^{m_1}f^{n}
e^{m_2}$ or $\phi=a^{r_1}b^{r_2}c^{r_3}d^{r_4}f^{n_1}e^{m}f^{n_2}$. Note that it is enough to study only one of these forms because of the symmetry between $e$ and $f$ given by rotation, so we will consider only the first one.

\begin{theorem} \label{thm:overtwisted}
The contact structure $\xi_{\phi}$ is overtwisted in the following
cases:
\begin{enumerate}
\item[(OT1)] $r_k <0$ for some $k$,
\item[(OT2)] $r_k=0$ for some $k$ and $min \{ m,n \}<0$,
\item[(OT3)] $min\{r_k\}=1$, \{$r_2=1$ or $r_4=1$\}, $min \{ m,n \}<0$ and $mn \geq 2$,
\item[(OT4)] $min\{r_k\}=1$, \{$r_1=1$ or $r_3=1$\}, $min \{ m,n \}<0$ and $mn \geq 2$,
\end{enumerate}
where $\phi=a^{r_1}b^{r_2}c^{r_3}d^{r_4}e^{m_1}f^{n}
e^{m_2} \in$ \au and $m=m_1+m_2$.
\end{theorem}

\medskip \noindent {\em Acknowledgments.\/} The authors would like to
thank Selman Akbulut, \c{C}a\u{g}r\i{} Karakurt, Yank\i{} Lekili, Burak \" Ozba\u{g}c\i, and Andr\'as
Stipsicz  for helpful conversations and remarks.
%
%
\section{Preliminaries} \label{sec:preliminaries}

First, we state the following classical fact which will be used in
Section \ref{sec:4-Holed}. We also give a proof since the authors
couldn't find the given version of the theorem in the literature.

\begin{theorem} \label{thm:Conjugates_Are_Contactomorphic}
Let S be any surface with nonempty boundary, and let $\sigma, h \in$
Aut$(S,\partial S)$. Then there exists a contactomorphism
$$(M_{(S,h)},\xi_{(S,h)}) \cong (M^{\prime}_{(S,\sigma h \sigma^{-1})},\xi^{\prime}_{(S,\sigma h \sigma^{-1})}). $$
\end{theorem}

\begin{proof}
The proof based on idea of breaking up the monodromy $\sigma h
\sigma^{-1}$ into pieces as depicted in Figure
\ref{BreakingMonodromyIntoPieces}. First take each glued solid torus
(around each binding component) out from both
$(M_{(S,h)},\xi_{(S,h)})$ and $(M^{\prime}_{(S,\sigma h
\sigma^{-1})},\xi^{\prime}_{(S,\sigma h \sigma^{-1})})$ to get the
mapping tori $S(h)$ and $S(\sigma h \sigma^{-1})$. By breaking the
monodromy $\sigma h \sigma^{-1}$, the mapping torus $S(\sigma h
\sigma^{-1})= [0,1]\times S/(1,x)\sim (0, \sigma h \sigma^{-1}(x))$
can be constructed also as follows: We write
\[S(\sigma h \sigma^{-1})=(\coprod_{i=1}^{4} S_i)/\sim,\]
where $S_i=S\times[\frac{i-1}{4}, \frac{i}{4}]$ and $\sim$ is the
equivalence relation that glues $S\times\{\frac{1}{4}\}$ in $S_1$ to
$S\times\{\frac{1}{4}\}$ in $S_2$ by $\sigma$, glues
$S\times\{\frac{1}{2}\}$ in $S_2$ to $S\times\{\frac{1}{2}\}$ in
$S_3$ by $h$, glues $S\times\{\frac{3}{4}\}$ in $S_3$ to
$S\times\{\frac{3}{4}\}$ in $S_4$ by $\sigma^{-1}$, glues
$S\times\{1\}$ in $S_4$ to $S\times\{0\}$ in $S_1$ by the identity
map $id$. (See the picture on the left in Figure
\ref{BreakingMonodromyIntoPieces}.)

\begin{figure}[ht]
  \begin{center}
   \includegraphics{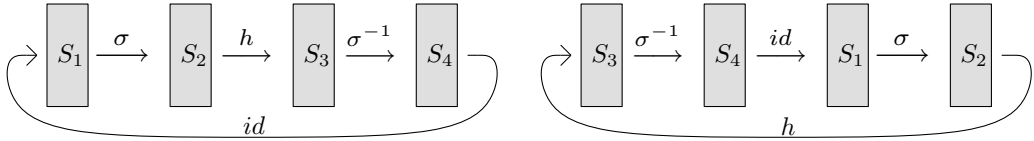}
   \caption{Mapping torus $S(\sigma h \sigma^{-1})$, before and after the cyclic permutation.}
  \label{BreakingMonodromyIntoPieces}
    \end{center}
\end{figure}

Since $S(\sigma h \sigma^{-1})$ is a fiber bundle over the circle
$S^1$, we are free to change its monodromy by any cyclic
permutation. Therefore, the monodromy element $\sigma^{-1}\cdot id
\cdot \sigma h=h $ also gives the same fiber bundle $S(\sigma h
\sigma^{-1})$ (the picture on the right in Figure
\ref{BreakingMonodromyIntoPieces} shows the new configuration of
$S(\sigma h \sigma^{-1})$ after the cyclic permutation). Therefore,
$S(h)=S(\sigma h \sigma^{-1})$. By gluing all solid tori back using identity, we
conclude that $(M_{(S,h)},\xi_{(S,h)})$ is contactomorphic to
$(M^{\prime}_{(S,\sigma h \sigma^{-1})},\xi^{\prime}_{(S,\sigma h
\sigma^{-1})})$.
\end{proof}


A \emph{Stein manifold} of dimension four is a triple
$(X^4,J,\psi)$ where $J$ is a complex structure on $X$, $\psi: X
\rightarrow \R$, and the $2-$form $\omega_\psi=-d(d\psi \circ J)$
is non-degenerate. We say that $(M^3,\xi)$ is \emph{Stein
(holomorphically) fillable} if there is a Stein manifold
$(X^4,J,\psi)$ such that $\psi$ is bounded from below, $M$ is a
non-critical level of $\psi$, and $-(d\psi\circ J)$ is a contact
form for $\xi$. The following fact was first implied in \cite{LP},
and then in \cite{AO}. The version given below is due to Giroux
and Matveyev. For a proof, see \cite{OS}.

\begin{theorem}\label{Holofill=PosDehnTwist}
A contact structure $\xi$ on $M^3$ is holomorphically fillable if
and only if  $\xi$ is supported by some open book whose monodromy
admits a factorization into positive Dehn twists only.
\end{theorem}

\medskip
\noindent \emph{Right-veering Diffeomorphisms:} We recall the
right-veering diffeomorphisms originally introduced in \cite{HKM1}.
If $S$ is a compact oriented surface with $\partial S \neq
\emptyset$, the submonoid $Veer(S,\partial S)$ of \rv elements in
$Aut(S,\partial S)$ is defined as follows: Let $\alpha$ and $\beta$
be isotopy classes (relative to the endpoints) of properly embedded
oriented arcs $[0,1]\rightarrow S$ with a common initial point
$\alpha(0)=\beta(0)=x\in \partial S$. Let $\pi:\tilde S\rightarrow
S$ be the universal cover of $S$ (the interior of $\tilde S$ will
always be $\R^2$ since $S$ has at least one boundary component), and
let $\tilde x\in \partial \tilde S$ be a lift of $x\in \partial S$.
Take lifts $\tilde \alpha$ and $\tilde \beta$ of $\alpha$ and
$\beta$ with $\tilde \alpha(0) =\tilde \beta(0)=\tilde x$. $\tilde
\alpha$ divides $\tilde S$ into two regions -- the region ``to the
left'' and the region ``to the right''. We say that $\beta$ is {\em
to the right} of $\alpha$, denoted $\alpha \ge \beta$, if either
$\alpha=\beta$ (and hence $\tilde\alpha(1)=\tilde\beta(1)$), or
$\tilde\beta(1)$ is in the region to the right (Figure
\ref{BetaRightAlfa}).
\begin{figure}[ht]
  \begin{center}
   \includegraphics{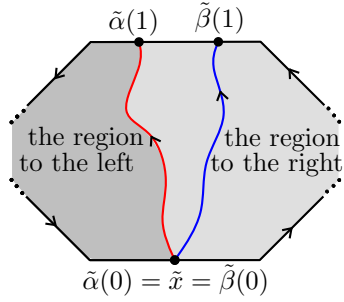}
   \caption{Lifts of $\alpha$ and $\beta$ in the universal cover $\tilde S$.}
  \label{BetaRightAlfa}
    \end{center}
\end{figure}

As an alternative way to passing to the universal cover, we first isotope
$\alpha$ and $\beta$, while fixing their endpoints, so that they
intersect transversely (including at the endpoints) and with the
fewest possible number of intersections. Then $\beta$ is to the
right of $\alpha$ if the tangent vectors $(\dot \beta(0),
\dot\alpha(0))$ define the orientation on $S$ at $x$.

\begin{definition}[\cite{HKM1}]
Let $h:S \to S$ be a diffeomorphism that restricts to the identity
map on $\partial S$. Let $\alpha$ be a properly embedded oriented
arc starting at a basepoint $x \in \partial S$. Then $h$ is {\em
right-veering} (that is, $h \in Veer(S,\partial S)$) if for every
choice of basepoint $x \in
\partial S$ and every choice of $\alpha$ based at $x$, $h(\alpha)$
is to the right of $\alpha$ (at $x$). If $C$ is a boundary component
of $S$, we say is $h$ is {\em right-veering with respect to $C$} if
$h(\alpha)$ is to the right of $\alpha$ for all $\alpha$ starting at
a point on $C$.
\end{definition}
It turns out that $Veer(S,\partial S)$ is a submonoid and we have
the inclusions:
\begin{center}
 $Dehn^{+}(S,\partial S) \subset Veer(S,\partial S) \subset Aut(S,\partial S)$.
\end{center}

We will use the following two results of \cite{HKM1}.
\begin{theorem}[\cite{HKM1}] \label{hokama}
A contact structure $(M,\xi)$ is tight if and only if all of its
compatible open book decompositions $(S,h)$ have right-veering $h
\in Veer(S,\partial S) \subset Aut(S,\partial S)$.
\end{theorem}

\begin{lemma}[\cite{HKM1}] \label{lem:subsurface}
Let $S$ be a hyperbolic surface with geodesic boundary and $\gamma \in
Aut(S,\partial S)$. Let $S'\subsetneq S$ be a subsurface, also with
geodesic boundary, and let $C$ be a common boundary component of $S$
and $S'$. If $\gamma$ is the identity map when restricted to $S'$,
$\delta$ is a closed curve parallel to and disjoint from $C$, and
$m$ is a positive integer, then $D^m_\delta \cdot \gamma$ is right-veering with
respect to $C$.
\end{lemma}

\begin{remark}
This lemma is useful in proving right-veering property and will be
used in the proof of Theorem \ref{thm:Phi_Right_Veering}. Note that in the case
of four$-$punctured sphere, $S'$ must have at least two holes to
apply Lemma \ref{lem:subsurface}
\end{remark}

%
%
\section{Right-Veering Diffeomorphisms and Overtwisted contact structures} \label{sec:ot}

In this section we will give the results which will be used to prove Theorem \ref{thm:overtwisted}.

\begin{lemma} \label{lem:twisting_OT}
Let $S$ be a planar hyperbolic surface with geodesic boundary
$\partial S = \cup_{i=1}^{l} {C_i}$, $l\ge 4$. Suppose $h \in
Aut(S,\partial S)$ and there is a properly embedded arc $\gamma$
starting at $x \in C_i$, ending at $C_j$ such that $h(\gamma)$ is to
the left of $\gamma$ at $x$ and $i \ne j$. Then $(h \cdot D_\delta)
(\gamma)$ is to the left of $\gamma$ at $x \in C_i$ for any curve
$\delta$ parallel to $C_k$ with $k\ne i$.
\end{lemma}

\begin{proof} Isotoping if necessary, we may assume that $\gamma$
and $h(\gamma)$ intersect minimally. We need to analyze two cases:

\noindent \textbf{Case 1.} Suppose $k\neq j$. Then we may assume
$\gamma \cap \delta =\emptyset$, and so $h(\gamma) \cap \delta =
\emptyset$. That is, $D_\delta$ fixes both $\gamma$ and
$h(\gamma)$. This implies that $D_\delta(h(\gamma))=h(\gamma)$ is to
the left of $\gamma$.

\noindent \textbf{Case 2.} Suppose $k=j$. First note that $h\neq
id_S$ since $h$ is not right-veering. Therefore, there exists a
region $R \subset S$ such that
\begin{enumerate}
\item $R$ is an embedded disk punctured $r-$times for some $0 < r < m-2$, and
\item $\partial R \subset \gamma \cup h(\gamma) \cup \partial S$.
\end{enumerate}
Let $C_{i_1}, \cdots, C_{i_r}$ be the common components of $\partial
S$ and $\partial R$. We may assume that $\partial R$ contains the
common initial point $x$ and the first intersection point $y$ (of
$\gamma$ and $h(\gamma)$) coming right after $x$ (See Figure
\ref{lemma_1_before}).

\begin{figure}[ht]
   \includegraphics{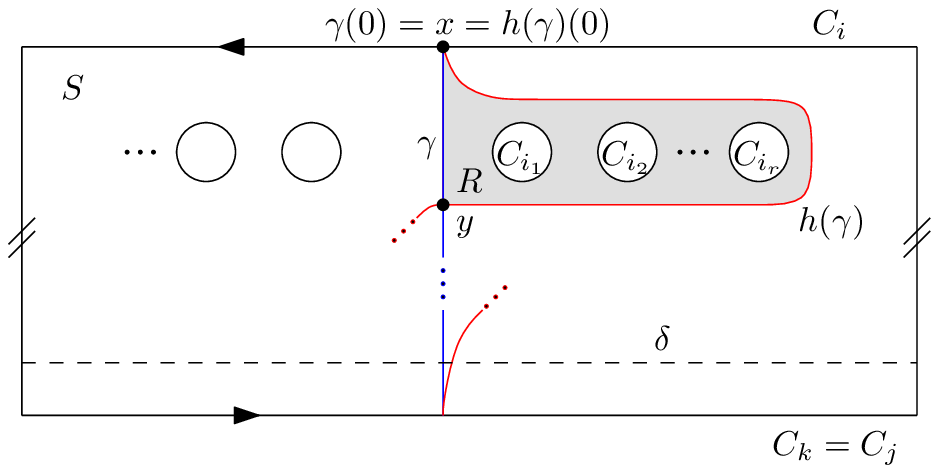}
   \caption{$h(\gamma)$ is to the left of $\gamma$ (left and right sides are identified).}
  \label{lemma_1_before}
\end{figure}

\begin{figure}[ht]
   \includegraphics{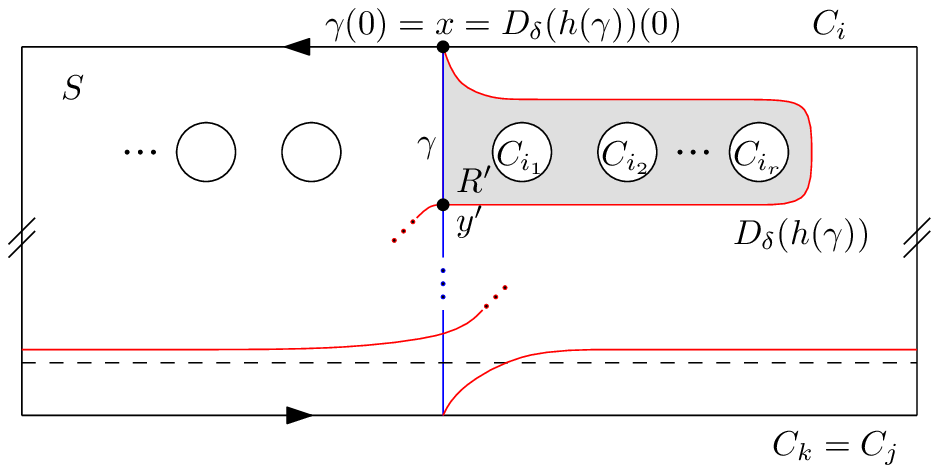}
   \caption{$D_\delta(h(\gamma))$ is to the left of $\gamma$ (left and right sides are identified).}
  \label{lemma_1_after}
\end{figure}

Since the Dehn twist $D_\delta$ is isotopic to the identity outside
of a small neighborhood of $\delta$, the image $R'=D_\delta(R)$ is
isotopic to $R$. In particular, $\partial R' \cap
D_\delta(h(\gamma))$ is to the left of $\partial R' \cap \gamma$
(see Figure \ref{lemma_1_after}). Note that $D_\delta(h(\gamma))$
and $\gamma$ are also intersecting minimally. Therefore, we conclude
that $(h \cdot D_\delta) (\gamma)=D_\delta(h(\gamma))$ is to the
left of $\gamma$.
\end{proof}

The following corollary of Lemma \ref{lem:twisting_OT} is immediate
with the help of Theorem \ref{hokama}.
\begin{corollary} \label{cor:twisting_OT}
Let $S$ be a planar hyperbolic surface with geodesic boundary
$\partial S = \cup_{i=1}^{l} {C_i}$, $l\ge 4$. Suppose $h \in
Aut(S,\partial S)$ is not right veering with respect to $C_i$ for
some $i$, and so the contact structure $\xi_{(S,h)}$ is overtwisted.
Then the contact structure $\xi_{(S,h\cdot{D_\delta}^k)}$ is also
overtwisted for any $k \in \Z_+$ and for any curve $\delta$ parallel
to the boundary component which is different than $C_i$. $\square$
\end{corollary}

Let us now interpret the notion of right-veering in terms of the
circle at infinity as in \cite{HKM1}. Let $S$ be any hyperbolic
surface with geodesic boundary $\partial S$. The universal cover
$\pi: \tilde S\rightarrow S$ can be viewed as a subset of the
Poincar\'e disk $D^2 = \mathbb{H}^2 \cup S^1_\infty$. Let $C$ be a
component of $\partial S$ and $L$ be a component of $\pi^{-1}(C)$.
If $h\in Aut(S,\partial S)$, let $\tilde h$ be the lift of $h$ that
is the identity on $L$. The closure of $\tilde S$ in $D^2$ is a
starlike disk. $L$ is contained in $\partial \tilde S$. Denote its
complement in $\partial \tilde S$ by $L_\infty$. Orient $L_\infty$
using the boundary orientation of $\tilde S$ and then linearly order
the interval $L_\infty$ via an orientation-preserving homeomorphism
with $\R$. The lift $\tilde h$ induces a homeomorphism
$h_\infty:L_\infty \to L_\infty$. Also, given two elements $a,b$ in
$Homeo^+(\R)$, the group of orientation-preserving homeomorphisms of
$\R$, we write $a\geq b$ if $a(z)\geq b(z)$ for all $z\in \R$ and
$a>b$ if $a(z)>b(z)$ for all $z\in \R$. In this setting, an element
$h$ is right-veering with respect to $C$ if $id\geq h_\infty$.
Equivalently,  if $\alpha$ is any properly embedded curve starting
at a point $\alpha(0) \in C$, and $\tilde{\alpha}$ is the lift of
$\alpha$ starting at the lift $\tilde{\alpha}(0) \in L$ of $x$, then
we have
\begin{center}
$h(\alpha)$ is to the right of $\alpha$ $\Longleftrightarrow$
$\tilde{\alpha}(1) \geq h_{\infty}(\tilde{\alpha}(1))$
\end{center}
Therefore, $h$ is not right-veering with respect to $C$ if there is
an arc $\alpha$ starting at $C$ such that we have $\tilde{\alpha}(1)
< h_{\infty}(\tilde{\alpha}(1))$.

\begin{lemma} \label{lem:OT+Left_twist}
Let $S$ be any hyperbolic surface with geodesic boundary $\partial
S$. Suppose $h \in Aut(S,\partial S)$ and there is a properly
embedded arc $\gamma$ starting at $x \in C \subset \partial S$ such
that $h(\gamma)$ is to the left of $\gamma$ at $x$. Then $(h \cdot
D_\sigma^{-1} )(\gamma)$ is to the left of $\gamma$ at $x \in C$ for
any simple closed curve $\sigma$ in $S$.
\end{lemma}

\begin{proof}
Write $\sigma$ for $D_\sigma$. Fix the identification of
$L_{\infty}$ with $\mathbb{R}$ as above. Consider the lift
$\tilde{\gamma}$ and induced homeomorphisms $h_\infty,
\sigma_\infty, \sigma^{-1}_\infty :L_\infty \to L_\infty$. Since
$\sigma^{-1}\cdot \sigma=id_S$, we have
$$(\sigma^{-1} \cdot \sigma)_\infty=\sigma_{\infty} \circ \sigma^{-1}_\infty=(id_S)_\infty.$$
Therefore, $\sigma^{-1}_\infty$ must map any point in $L_\infty$ to
its left because $\sigma$ is right-veering. In particular, $(h \cdot
\sigma^{-1})_\infty(\tilde{\gamma}(1))=\sigma^{-1}_\infty(h_\infty(\tilde{\gamma}(1)))$
is to the left of $h_\infty(\tilde{\gamma}(1))$ which is (by the
assumption) to the left of $\tilde{\gamma}(1)$. That is, $(h \cdot
\sigma^{-1})_\infty(\tilde{\gamma}(1))>h_\infty(\tilde{\gamma}(1))>\tilde{\gamma}(1)$.
\end{proof}

\begin{corollary} \label{cor:OT+Left_twist}
Let $S$ be a planar hyperbolic surface with geodesic boundary
$\partial S = \cup_{i=1}^{l} {C_i}$, $l\ge 4$. Suppose $h \in
Aut(S,\partial S)$ is not right veering with respect to $C_i$ for
some $i$, and so the contact structure $\xi_{(S,h)}$ is overtwisted.
Then the contact structure $\xi_{(S,h\cdot{D_\sigma}^k)}$ is also
overtwisted for any $k \in \Z_-$ and for any simple closed curve
$\sigma$ in $\Sigma$. $\square$
\end{corollary}

\begin{remark} \label{rem:Yilmaz}
The idea used in the proof of Lemma \ref{lem:OT+Left_twist} gives a
simple proof for Lemma 6 of \cite{Y}. Moreover, the following lemma
is given as Lemma 5 in \cite{Y}. We want to give a different proof
for it using the idea of the circle at infinity.
\end{remark}

\begin{lemma} \label{lem:Conjugate_of_RV_is_RV}
Let $S$ be a hyperbolic surface with geodesic boundary, and let
$h\in Aut(S,\partial S)$ be a right-veering diffeomorphism. Then
$h^{\prime}=\sigma h \sigma^{-1}$ is right-veering for any $\sigma
\in Aut(S,\partial S)$.
\end{lemma}

\begin{proof}  Clearly, it is enough to consider the case when $\sigma$
is a single Dehn twist. First, assume that $\sigma$ is a positive
Dehn twist. We need to show that $h^{\prime}$ is right-veering with
respect to any boundary component of $S$. We will use the notations
introduced in the previous paragraph. So fix the boundary component
$C$, and an identification of $L_{\infty}$ with $\mathbb{R}$ as
above. Let $\alpha$ be any properly embedded curve in $S$ starting
at a point $\alpha(0) \in C$. Consider the lift $\tilde{\alpha}$ and
induced homeomorphisms $h^{\prime}_\infty, h_\infty, \sigma_\infty,
\sigma^{-1}_\infty :L_\infty \to L_\infty$. From their definitions
we have
$$h^{\prime}_\infty(\tilde{\alpha}(1))=\tilde{h^{\prime}}(\tilde{\alpha}(1))
=\widetilde{\sigma h \sigma^{-1}}(\tilde{\alpha}(1))=\tilde{\sigma}
\tilde{h} \tilde{\sigma^{-1}}(\tilde{\alpha}(1))=\sigma_\infty
h_\infty \sigma^{-1}_\infty (\tilde{\alpha}(1))$$ Suppose that
$\sigma^{-1}_\infty (\tilde{\alpha}(1))=a \in L_\infty$ and
$h_\infty(a)=b \in L_\infty$. Then since
$$\sigma_\infty(b)=({(\sigma^{-1})}^{-1})_\infty(b)={(\sigma^{-1}_\infty)}^{-1}(b),$$
$b$ must be mapped (by $\sigma_\infty$) to a point in $L_\infty$
which is to the right of $\tilde{\alpha}(1)$ as we illustrated in
Figure \ref{ConjugateIsRV}.

\begin{figure}[ht]
   \includegraphics{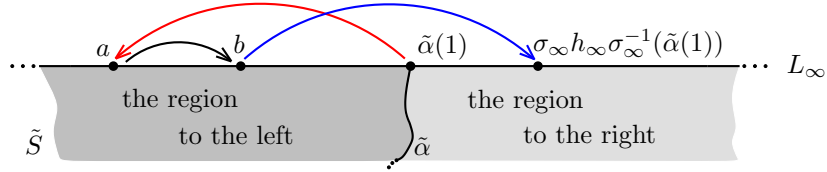}
   \caption{The point $\tilde{\alpha}(1) \in L_\infty \approx \mathbb{R}$, and how it is mapped to the right of itself.}
  \label{ConjugateIsRV}
\end{figure}

Equivalently, $\tilde{\alpha}(1) \geq \sigma_\infty h_\infty
\sigma^{-1}_\infty(\tilde{\alpha}(1))=h^{\prime}_\infty(\tilde{\alpha}(1))$
implying that $h^{\prime}$ is right-veering with respect to $C$. The
proof of the case when $\sigma$ is a negative Dehn twist uses exactly
the same argument, so we omit it.
\end{proof}

%
%
\section{Four$-$Punctured Sphere and the Proofs of Main Theorems} \label{sec:4-Holed}

For simplicity, we will denote the Dehn twist along any simple
closed curve by the same letter we use for that curve.
\begin{definition}
A representative of an element $\phi \in$ \au is said to be in \textbf{reduced form} if $s$ is the smallest integer such that $\phi$ can be written as
$$\phi=a^{r_1}b^{r_2}c^{r_3}d^{r_4}e^{m_1}f^{n_1}e^{m_2}f^{n_2}\cdots  e^{m_{s-1}}f^{n_{s-1}}e^{m_s}f^{n_s}$$
where $r_k, m_i, n_i$ are all integer for $1\leq k \leq 4$, $1\leq i
\leq s$ with possibly $m_1$ or $n_s$ zero.
\end{definition}
\begin{lemma} \label{lem:reducing_momodromy}
Any element $\phi \in$ \au can be written in reduced form.
\end{lemma}
\begin{proof}
From braid group representation of full mapping class group, we know
that the mapping class group \au can be generated by Dehn twists
along the simple closed curves $a, b, c, d, e, f, g, h$ given in
Figure \ref{4-holed_1} (see \cite{Bi} for details). Therefore, any
element $\phi$ of \au can be written as a word consisting of only
$a, b, c, d, e, f, g, h$ and their inverses. Since $a, b, c, d,$ are
in the center of \aut, we can bring them to any position we want.
For the second part including $e$ and $f$, we use the well-known
\emph{lantern relation} (also known as \emph{4-holed sphere
relation}). In terms of our generators we will use two different
lantern relations. Namely, we have
\begin{center}
$gef=abcd \quad \quad$ and $\quad \quad hfe=abcd$.
\end{center}
These give $g=abcdf^{-1}e^{-1}$ and $h=abcde^{-1}f^{-1}$. Therefore,
we can exchange any power of $g$ and $h$ in the word defining $\phi$
by some products of $a, b, c, d, e^{-1}, f^{-1}$ (and $a^{-1},
b^{-1}, c^{-1}, d^{-1}, e, f$ for negative powers of $g$ and $h$).
Combining (and canceling if there is any) the powers of $e$ and $f$,
and commuting the generators $a, b, c, d$, we get the reduced form
of $\phi$ as claimed.
\end{proof}

We first prove Theorem \ref{thm:holofillable} using the lantern relations.

\begin{proof}[\textbf{Proof of Theorem \ref{thm:holofillable}}]
Let
$\phi=a^{r_1}b^{r_2}c^{r_3}d^{r_4}e^{m_1}f^{n_1}\cdots
e^{m_s}f^{n_s} \in $ \au. We will show how to obtain a monodromy for the same open book which is a product of positive Dehn twists and use Theorem \ref{Holofill=PosDehnTwist}. Using lantern relations, we can replace each $e^{-1}$ by $a^{-1}b^{-1}c^{-1}d^{-1}hf$ and each
$f^{-1}$ by $a^{-1}b^{-1}c^{-1}d^{-1}ge$. This proves (H1) and (H4). In the case where $s=1$, we can use fewer lantern relations by first doing the following: using the lantern relations, replace $e^{-1}f^{-1}$ by $a^{-1}b^{-1}c^{-1}d^{-1}h$. Moreover, if $max\{m_1,n_1\}<-1$, also use the lantern relations to replace $e^-1$ by $a^{-1}b^{-1}c^{-1}d^{-1}fg$, and use Theorem \ref{thm:Conjugates_Are_Contactomorphic} to cancel the new initial $f$ with the last $f^{-1}$.
\end{proof}

We note that the two simplifications mentioned in the case of $s=1$ can also be applied in general, but negative powers of $e$ and $f$ need not be adjacent (even after a cyclic permutation).
We also remark that changing the order of the products of $e$ and $f$ can result in not only different contact manifolds, but also topologically different manifolds. For instance, $\phi=e^2f^2$ and $\phi=efef$ are not conjugate to each other, and the underlying topological manifolds are not diffeomorphic (see Figure \ref{efef_eeff} for their surgery diagrams).

\begin{figure}[ht]
   \includegraphics[width=3in]{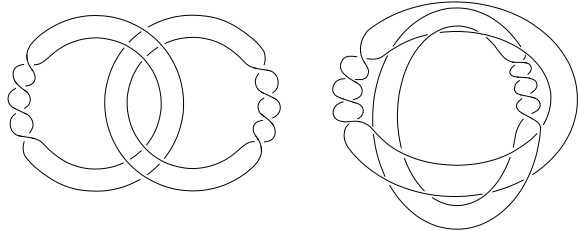}
   \caption{Diagrams for $\phi=e^2f^2$ and $\phi=efef$ (all coefficients are $-3$).}
  \label{efef_eeff}
\end{figure}

\newpage
Next, we characterize the overtwisted structures stated in the introduction.

\begin{proof}[\textbf{Proof of Theorem \ref{thm:overtwisted}}]
By using Theorem \ref{thm:Conjugates_Are_Contactomorphic}, we will prove the statements for
$\xi_{\phi^{\prime}}$ where
$\phi^{\prime}=a^{r_1}b^{r_2}c^{r_3}d^{r_4}e^mf^n$. To prove (OT1),
consider the properly embedded curves $\alpha_1, \alpha_2, \alpha_3,
\alpha_4$ starting at the boundary components $C_1, C_2, C_3, C_4$,
respectively, and their images under $\phi^{\prime}$ as given in
Figure \ref{overtwisted1}.
\begin{figure}[ht]
   \includegraphics{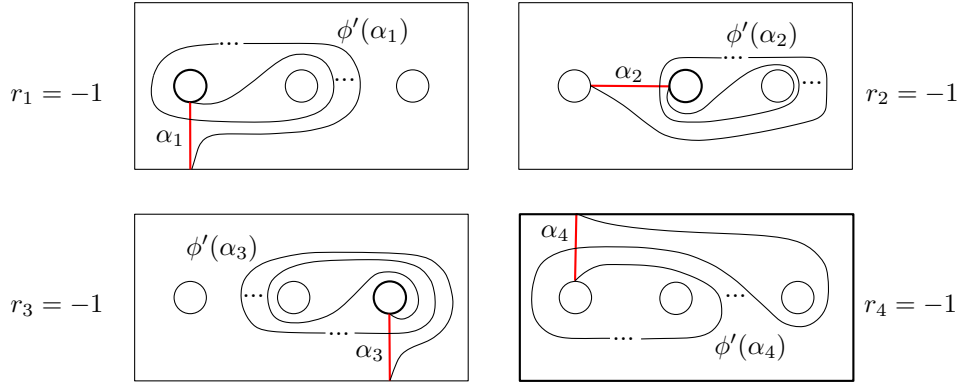}
   \caption{The curves $\alpha_k$ and their images under $\phi^{\prime}$ in $\Sigma$.}
  \label{overtwisted1}
\end{figure}
In all the pictures, we are assuming $m>0$, $n>0$, and $r_k=-1$
(otherwise the fact that $\phi^{\prime}$ is not right-veering with
respect to $C_k$ is even more obvious). We can see from the pictures
that if $r_k<0$ for some $k$, then $\phi^{\prime}(\alpha_k)$ is to
the left of $\alpha_k$ , so $\phi^{\prime}$ is not right-veering
which implies by Theorem \ref{hokama} that $\xi_{\phi^{\prime}}$ is
overtwisted. Note that in any picture in Figure \ref{overtwisted1},
we are taking all the other $r_k$'s to be zero. However, even if
$\phi^{\prime}$ has a factor of some positive power of Dehn twist
along the boundary component other than $C_k$,
$\phi^{\prime}(\alpha_k)$ is still left to the $\alpha_k$ at their
common starting point by Lemma \ref{lem:twisting_OT}. Therefore,
$\xi_{\phi'}$ is overtwisted by Corollary \ref{cor:twisting_OT}.

\begin{figure}[ht]
   \includegraphics{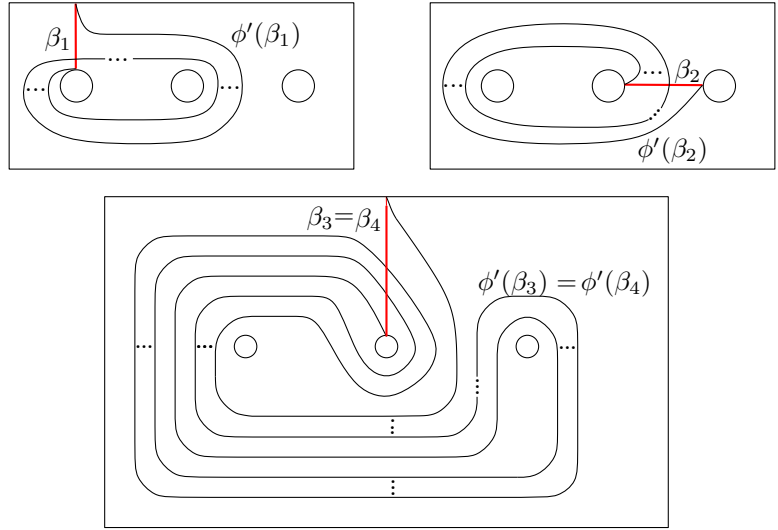}
   \caption{The curves $\beta_k$ and their images under $\phi^{\prime}$ in $\Sigma$.}
  \label{overtwisted2}
\end{figure}

To prove (OT2), consider the properly embedded curves $\beta_1,
\beta_2, \beta_3, \beta_4$ starting at the boundary components
$C_1, C_2, C_3, C_4$, respectively, and their images under
$\phi^{\prime}$ as given in
Figure \ref{overtwisted2}. In all the pictures, we are
assuming $m=-1$, $n>0$, (again otherwise the fact that
$\phi^{\prime}$ is not right-veering with respect to $C_k$ is even
more obvious). We can see from the pictures that if $r_k=0$ for
some $k$, then $\phi^{\prime}(\beta_k)$ is to the left of
$\beta_k$ , so $\phi^{\prime}$ is not right-veering which implies
again by Theorem \ref{hokama} that $\xi_{\phi^{\prime}}$ is
overtwisted. Again, in all the pictures, we consider all the other
$r_k$'s to be zero, and if $\phi^{\prime}$ has a factor of some
positive power of Dehn twist along the boundary component other
than $C_k$, $\phi^{\prime}(\beta_k)$ is still left to the
$\beta_k$ at their common starting point by Lemma
\ref{lem:twisting_OT}. Therefore, $\xi_{\phi'}$ is overtwisted by
Corollary \ref{cor:twisting_OT}.

\begin{figure}[ht]
   \includegraphics{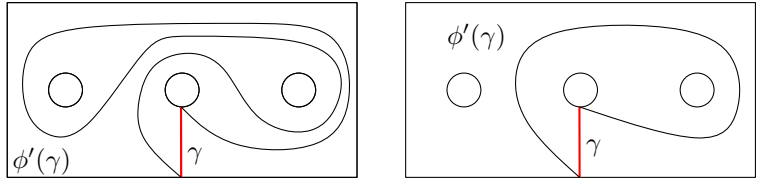}
   \caption{The curve $\gamma$ and its images under two possible $\phi^{\prime}$ in $\Sigma$.}
  \label{overtwisted3}
\end{figure}

To prove (OT3), consider the curve $\gamma$ running from $C_2$ to
$C_4$ as in Figure \ref{overtwisted3}. In the left picture each
$r_k=1, m=-2, n=-1$, and in the right one each $r_k=1, m=-1,
n=-2$. Clearly, the image $\phi'(\gamma)$ is to left of $\gamma$
at both their common endpoints on $C_2$ and $C_4$. Therefore,
$\xi_{\phi'}$ ($\phi'=abcde^{-2}f^{-1}$ or $abcde^{-1}f^{-2}$) is
overtwisted. In both cases, if we take $r_1, r_3$ and only one of
$r_2$ and $r_4$ to be any positive integer, $\xi_{\phi'}$ is still
overtwisted by Lemma \ref{lem:twisting_OT} and Corollary
\ref{cor:twisting_OT}. Moreover, if we also take $m\leq -3, n\leq
-3$ in both cases, $\xi_{\phi'}$ is still overtwisted by Lemma
\ref{lem:OT+Left_twist} and Corollary \ref{cor:OT+Left_twist}.

The proof of (OT4) is similar to that of (OT3), so we will omit
it.

\end{proof}

%
%

The main trick we have used through out the paper is
that we proved most of the the statements for
$\phi^{\prime}=a^{r_1}b^{r_2}c^{r_3}d^{r_4}e^mf^n$ and then applied
Theorem \ref{thm:Conjugates_Are_Contactomorphic}. For the cases
which we could not decide whether $\xi_{\phi^{\prime}}$ is overtwisted or not,
it is still good to know if
$\phi^{\prime}$ is right-veering. We give some partial answers to
this in the next theorem where we do not list some obvious (or
already proven) cases.

\begin{theorem} \label{thm:Phi_Right_Veering}
$\phi^{\prime}=a^{r_1}b^{r_2}c^{r_3}d^{r_4}e^mf^n \in$ \au is
right-veering in the following cases:
\begin{enumerate}
\item[(R1)] $min\{r_k\}=1$, $mn=0$, and $max\{m,n\}=0$,
\item[(R2)] $min\{r_k\}=1$, and $mn < 0$.
\end{enumerate}
\end{theorem}

\begin{proof}
For (R1), assume $m<0$ and $n=0$. The fact that $\phi'$ is
right-veering is an implication of Lemma \ref{lem:subsurface} as
follows: To show $\phi'$ is right-veering with respect to $C_1$
and $C_2$, take $S'$ (in the lemma) to be the subsurface of
$\Sigma$ such that $\partial S'=C_1 \cup C_2 \cup e$ and take
$\gamma$ (in the lemma) as $\gamma=c^{r_3}d^{r_4}e^m$. To show
$\phi'$ is right-veering with respect to $C_3$ and $C_4$, take
$S'$ to be the subsurface of $\Sigma$ such that $\partial S'=C_3
\cup C_4 \cup e$ and take $\gamma$ as $\gamma=a^{r_1}b^{r_2}e^m$.

For (R2), assume $m<0$ and $n>0$. We know, by (R1), that
$\tilde{\phi}=a^{r_1}b^{r_2}c^{r_3}d^{r_4}e^m$ ($r_k \geq 1$ for
all $k$, $m < 0$) is right-veering. Since $Dehn^+(\Sigma, \partial
\Sigma) \subset Veer(\Sigma, \partial \Sigma)$, we conclude that
$\phi'=\tilde{\phi} \cdot f^n$ ($n>0$) is also right-veering.
\end{proof}


\end{document}